\documentclass[12pt]{amsart}
\usepackage{url,mathrsfs,amsmath,amssymb,color,amsfonts,amsthm,latexsym,marginnote}
\theoremstyle{plain}
\newtheorem{theorem}{Theorem}[section]
\newtheorem{proposition}[theorem]{Proposition}
\newtheorem{corollary}[theorem]{Corollary}
\newtheorem{lemma}[theorem]{Lemma}

\theoremstyle{definition}
\newtheorem{remark}[theorem]{Remark}
\newtheorem{example}[theorem]{Example}

\newcommand{\abs}[1]{\lvert#1\rvert}
\newcommand{\norm}[1]{\lVert#1\rVert}

\newcommand{\term}[1]{{\textit{\textbf{#1}}}}

\title[Unbounded order convergence]{Unbounded Order Convergence in Dual Spaces}
\date{}
\author[N.~Gao]{Niushan Gao}
\address{Department of Mathematical and Statistical Sciences, University of Alberta, Edmonton, AB, Canada T6G\,2G1}
\email{niushan@ualberta.ca}

\begin{document}
\maketitle
\begin{abstract}A net $(x_\alpha)$ in a vector lattice $X$ is said to be {unbounded order convergent} (or uo-convergent, for short) to $x\in X$
if the net $(\abs{x_\alpha-x}\wedge y)$ converges to $0$ in order for all $y\in X_+$. In this paper, we study unbounded order convergence in dual spaces of Banach lattices. Let $X$ be a Banach lattice. We prove that every norm bounded uo-convergent net in $X^*$ is $w^*$-convergent iff $X$ has order continuous norm, and that every $w^*$-convergent net in $X^*$ is uo-convergent iff $X$ is atomic with order continuous norm. We also characterize among $\sigma$-order complete Banach lattices the spaces in whose dual space every simultaneously uo- and $w^*$-convergent sequence converges weakly/in norm.
\end{abstract}

\section{Introduction}
This paper is a continuation of \cite{gx}. We follow the terminology and notations from \cite{gx}. Recall that a net $(x_\alpha)$ in a vector lattice $X$ is \term{unbounded order convergent} (or uo-convergent, for short) to $x\in X$ if $\abs{x_\alpha-x}\wedge y\xrightarrow{o}0$ for all $y\in X_+$. In this case, we write $x_\alpha\xrightarrow{uo}x$. It is easily seen that a sequence $(x_n)$ in $L_1(\mu)$ uo-converges to $x\in L_1(\mu)$ iff $(x_n)$ converges to $x$ almost everywhere. Let $\mathbb{R}^A$ be the vector lattice of all real-valued functions on a non-empty set $A$, equipped with the pointwise order. It is easily seen that
a net $(x_\alpha)$ in $\mathbb{R}^A$ uo-converges to $x\in \mathbb{R}^A$ iff it converges pointwise to $x$.

The study of uo-convergence was initiated in \cite{nak,dem64}. In \cite{wi77}, Wickstead initiated the study of relations between uo-convergence and topological properties of the underlying spaces. He characterized the spaces in which uo-convergence of nets implies weak convergence and vice versa. In \cite{gx}, Xanthos and the author studied nets which simultaneously have weak and uo convergence properties, and characterized among $\sigma$-order complete Banach lattices the spaces in which every weakly and uo-convergent sequence is norm convergent.\par

In this paper, we study uo-convergence in dual spaces. Let $X$ be a Banach lattice. We are motivated by \cite{wi77} and \cite{gx} to consider the following:
\begin{enumerate}
\item characterize the spaces $X$ such that in its dual space $X^*$, uo-convergence implies $w^*$-convergence and vice versa
\item study nets/sequences in $X^*$ which simultaneously have uo- and $w^*$-convergence properties, and characterize the spaces in whose dual space simultaneous uo- and $w^*$-convergence imply weak/norm convergence.
\end{enumerate}
\medskip
\hspace*{\parindent}The remark at the end of this paper suggests the notion of uo-convergence in dual spaces could serve as a tool for the study of geometry of Banach lattices. We first remark here a few useful facts about uo-convergence.
\begin{lemma}\label{sim}Let $X$ be a $\sigma$-order complete vector lattice and $(x_n)$ a disjoint sequence in $X$. Then $(x_n)$ uo-converges to $0$ in $X$.
\end{lemma}
\begin{proof}Fix $x\in X_+$. We claim that $\sup_{k\geq n}(\abs{x_k}\wedge x)\downarrow_n 0$. Indeed, suppose $\sup_{k\geq n}(\abs{x_k}\wedge x)\geq y\geq 0$ for all $n\geq 1$. Then
\begin{align*}
0\leq y\wedge \abs{x_n}\leq &\left(\sup_{k\geq n+1}(\abs{x_k}\wedge x)\right)\wedge \abs{x_{n}}=\sup_{k\geq n+1}(\abs{x_k}\wedge \abs{x_n}\wedge y)=0.
\end{align*} Thus, $y\wedge \abs{x_n}=0$ for all $n\geq1$. It follows that $y=y\wedge \sup_{n\geq 1}(\abs{x_n}\wedge x)=\sup_{n\geq 1}(y\wedge \abs{x_n}\wedge x)=0$.
This proves the claim. Now it is immediate that $\abs{x_n}\wedge x\xrightarrow{o}0$.
\end{proof}

Recall that a vector $x>0$ in a vector lattice $X$ is called an {atom} if the ideal generated by $x$ is one-dimensional, and that a vector lattice $X$ is said to be atomic if the linear span of all atoms is order dense in $X$.
\begin{lemma}\label{sim2}\begin{enumerate}
\item\label{sim2i1} For a sequence $(x_n)$ in a vector lattice $X$, if $x_n\xrightarrow{uo}0$, then $\inf_k\abs{x_{n_k}}=0$ for any increasing sequence $(n_k)$ of natural numbers.
\item\label{sim2i2} The converse holds true if $X$ is an order complete atomic vector lattice.
\end{enumerate}
\end{lemma}
\begin{proof}\eqref{sim2i1} Suppose $x_n\xrightarrow{uo}0$. Fix any increasing sequence $(n_k)$ of natural numbers. It is clear that $x_{n_k}\xrightarrow{uo}0$. Let $x\in X$ be such that $\abs{x_{n_k}}\geq x\geq 0$ for all $k\geq 1$. Then $x=\abs{x_{n_k}}\wedge x\xrightarrow{o}0$, implying that $x=0$. Hence, $\inf_k\abs{x_{n_k}}=0$.\par

\eqref{sim2i2} Suppose $x_n\not\xrightarrow{uo}0$ in $X$. Assume first that $X$ is an order complete atomic vector lattice. Then $X$ embeds as an ideal into $\mathbb{R}^A$ for some non-empty set $A$; cf.~\cite[Exercise~7, p.~143]{scha3}. It follows from \cite[Lemma~3.4]{gx} that $x_n\not\xrightarrow{uo}0$ in $\mathbb{R}^A$. Therefore, $(x_n(a))$ does not converge to $0$ for some point $a\in A$. Take an increasing sequence $(n_k)$ such that $\inf_k\abs{x_{n_k}(a)}>0$. Then $\inf_k\abs{x_{n_k}}>0$ in $\mathbb{R}^A$, and thus, in $X$.
\end{proof}

\section{When does uo-convergence imply $w^*$-convergence?}
Observe that the spaces in whose dual space every uo-convergent sequence is $w^*$-convergent are finite dimensional. Indeed, otherwise, assume $\dim X=\infty$. Then we can take a normalized disjoint sequence $(x_n^*)$ in $X^*$. By Lemma~\ref{sim}, $(nx_n^*)$ uo-converges to $0$ in $X^*$. Thus, $(nx_n^*) $ is $w^*$-convergent, by hypothesis. In particular, $(nx_n^*) $ is norm bounded, which is absurd. So the ``correct'' question one may ask is when uo-convergence of norm bounded nets/sequences implies $w^*$-convergence. The theorem below answers this question.

\begin{theorem}\label{u2w}Let $X$ be a Banach lattice. The following are equivalent:
\begin{enumerate}
\item\label{u2wi1} $X$ has order continuous norm,
\item\label{u2wi2} for any norm bounded net $(x_\alpha^*)$ in $X^*$, if $x_\alpha^*\xrightarrow{uo}0$, then $x_\alpha^*\xrightarrow{w^*}0$,
\item\label{u2wi2.5} for any norm bounded net $(x_\alpha^*)$ in $X^*$, if $x_\alpha^*\xrightarrow{uo}0$, then $x_\alpha^*\xrightarrow{\abs{\sigma}(X^*,X) }0$,
\item\label{u2wi3} for any norm bounded sequence $(x_n^*)$ in $X^*$, if $x_n^*\xrightarrow{uo}0$, then $x_n^*\xrightarrow{w^*} 0$,
\item\label{u2wi3.5} for any norm bounded sequence $(x_n^*)$ in $X^*$, if $x_n^*\xrightarrow{uo}0$, then $x_n^*\xrightarrow{\abs{\sigma}(X^*,X) }0$.
\end{enumerate}
\end{theorem}

\begin{proof}
Since $x_\alpha^*\xrightarrow{uo}0$ iff $\abs{x_\alpha^*}\xrightarrow{uo}0$, it is clear that \eqref{u2wi2}$\Leftrightarrow$\eqref{u2wi2.5} and \eqref{u2wi3}$\Leftrightarrow$\eqref{u2wi3.5}.\par

Let $X$ be order continuous. Suppose $(x_\alpha^*)\subset X^*$ is norm bounded and $x_\alpha^*\xrightarrow{uo}0$. Pick any $x\in X_+$. For any $\varepsilon>0$, we have, by \cite[Theorem~4.18]{ali}, there exists $y^*\in X_+^*$ such that $\abs{x^*_\alpha}(x)-\abs{x^*_\alpha}\wedge y^*(x)<\varepsilon$ for all $\alpha$.
Since $\abs{x^*_\alpha}\wedge y^*\xrightarrow{o}0$, we have $\abs{x^*_\alpha}\wedge y^*(x)\rightarrow 0$. Thus, $\limsup_\alpha\abs{x^*_\alpha}(x)\leq \varepsilon$. By arbitrariness of $\varepsilon$, we have $\lim_\alpha\abs{x^*_\alpha}(x)=0$.
This proves \eqref{u2wi1}$\Rightarrow$\eqref{u2wi2.5}.\par
The implication \eqref{u2wi2.5}$\Rightarrow$\eqref{u2wi3.5} is clear. We now prove \eqref{u2wi3.5}$\Rightarrow$\eqref{u2wi1}. Assume \eqref{u2wi3.5} holds. Take any norm bounded disjoint sequence $(x_n^*)$ in $X^*$. By Lemma~\ref{sim}, $x_n^*\xrightarrow{uo}0$. Thus, $x_n^*\xrightarrow{w^*}0$ by assumption. It follows from \cite[Corollary~2.4.3]{mn1991} that $X$ is order continuous.
\end{proof}

We apply this result to establish the following dual version of \cite[Theorem~4.3]{gx}. Recall that a net $(x_\alpha)$ in a vector lattice $X$ is \term{unbounded order Cauchy} (or uo-Cauchy, for short), if the net $(x_\alpha-x_{\alpha'})_{(\alpha,\alpha')}$ uo-converges to $0$ in $X$.

\begin{theorem}\label{wstar}
Let $X$ be an order continuous Banach lattice. Then any norm bounded uo-Cauchy net in $X^*$ converges in uo and $\abs{\sigma}(X^*,X)$ to the same limit.
\end{theorem}

\begin{proof}Let $(x_\alpha^*)\subset X^*$ be norm bounded and uo-Cauchy in $X^*$. It follows from Theorem~\ref{u2w} that $(x_\alpha^*)$ is $w^*$-Cauchy. Since $(x_\alpha^*)$ is norm bounded, we have $x_\alpha^*\xrightarrow{w^*}x^*$ for some $x^*\in X^*$. By Theorem~\ref{u2w} again, it suffices to show that $x_\alpha^*\xrightarrow{uo}x^*$.\par

We first assume that $X$ has a weak unit $x_0>0$. Then $x_0$ acts a strictly positive order continuous functional on $X^*$. Let $\widetilde{X^*}$ be the AL-space constructed for the pair $(X^*,x_0)$; cf.~\cite[Subsection~2.2]{gx}. Then $\widetilde{X^*}^*$ is lattice isomorphic to $I_{x_0}$, the ideal generated by $x_0$ in $X^{**}$. Since $X$ is order continuous, it is an ideal in $X^{**}$. Therefore, $I_{x_0}$ is the same as the ideal generated by $x_0$ in $X$. Thus, since $\{x_\alpha^*:\alpha\})$ is relatively $w^*$-compact in $X^*$, it is relatively $w$-compact in $\widetilde{X^*}$. By \cite[Lemma~2.2]{gx}, $(x_\alpha^*)$ is also uo-Cauchy in $\widetilde{X^*}$. It follows from \cite[Theorem~4.3]{gx} that $(x^*_\alpha)$ converges uo and weakly to the same limit in $\widetilde{X^*}$. Since $(x_\alpha^*)$ converges weakly to $x^*$ in $\widetilde{X^*}$, we have $(x^*_\alpha)$ uo-converges to $x^*$ in $\widetilde{X^*}$, and therefore, in $X^*$ by \cite[Lemma~2.2]{gx} again. This proves the special case.\par

We now prove the general case. Let $\{y_{\gamma} :\gamma\in\Gamma\}$ be a maximal collection of pairwise disjoint elements of $X$. Let $\Delta$ be the collection of all finite subsets of $\Gamma$ directed by inclusion. For each $\delta=\{\gamma_1,\dots,\gamma_n\}\in\Delta$, the band $B_\delta$ generated by $\{y_{\gamma_i}\}_1^n$ is an order continuous Banach lattice with a weak unit $y_\delta=\sum_1^ny_{\gamma_i}$. Let $P_\delta$ be the band projection for $B_\delta$. Then $P_\delta x\rightarrow x$ for each $x\in X_+$. It follows that $P_\delta^*x^*\uparrow x^*$ for each $x^*\in X_+^*$.\par

Fix any $y^*\in X_+^*$. Since $(x_\alpha^*)$ is uo-Cauchy, there exists a net $(x_{\alpha,\alpha'}^*)\subset X_+^*$ such that $\abs{x_\alpha^*-x_{\alpha'}^* }\wedge y^*\leq x_{\alpha,\alpha'}^*\downarrow0$. Fix any $\delta$.
Observe that $(x_\alpha^*|_{B_\delta})$ is norm bounded and $w^*$-converges to $x^*|_{B_\delta}$ in $B^*_{\delta}$. Since $X=B_\delta\oplus B_\delta^d$, it is easily seen that $(x_\alpha^*|_{B_\delta})$ is also uo-Cauchy in $B^*_{\delta}$. Thus by the preceding paragraph, we have $(x_\alpha^*|_{B_\delta})$ uo-converges to $x^*|_{B_\delta}$ in $B_\delta^*$. It follows that $P_\delta^*\big(\abs{x_\alpha^*-x^* }\wedge y^*\big)=\abs{P_\delta^* x_\alpha^*-P_\delta^* x^*}\wedge P_\delta^*y^*\xrightarrow{o}_\alpha0$ in $X^*$.
Note that
\begin{align*}P_\delta^*\big(\abs{x_\alpha^*-x^* }\wedge y^*\big)\leq& P_\delta^*\big(\abs{x_\alpha^*-x_{\alpha'}^* }\wedge
y^*\big)+P_\delta^*\big(\abs{x_{\alpha'}^*-x^* }\wedge y^*\big)\\
\leq & x_{\alpha,\alpha'}^*+P_\delta^*\big(\abs{x_{\alpha'}^*-x^* }\wedge y^*\big).\end{align*}
Taking $\alpha'\rightarrow\infty $, we have $P_\delta^*\big(\abs{x_\alpha^*-x^* }\wedge y^*\big)\leq \inf_{\alpha'}x_{\alpha,\alpha'}^*$. Now taking $\delta\rightarrow\infty$, we have $\abs{x_\alpha^*-x^* }\wedge y^*\leq\inf_{\alpha'}x_{\alpha,\alpha'}^*\downarrow_\alpha0$ in $X^*$. This proves $x_{\alpha}^*\xrightarrow{uo}x^*$ in $X^*$.
\end{proof}

We now apply Theorem~\ref{wstar} to study abstract (sub-)martingales in dual spaces. We refer the reader to \cite[Section~5]{gx} for the terminology. The following lemma is well known and follows from Nakano's Theorem (\cite[Theorem~1.67]{ali}).

\begin{lemma}\label{spf}Let $X$ be an order continuous Banach lattice. Then a positive functional $x_0^*>0$ is a weak unit of $X^*$ iff it is strictly positive on $X$.
\end{lemma}

\begin{corollary}\label{mar}Let $X$ be an order continuous Banach lattice with a weak unit. Suppose that $(E_n)$ is a bounded filtration on $X$ such that both $E_1$ and $E_1^*$ are strictly positive. Then every norm bounded submartingale relative to $(E_n^*)$ converges in uo and $\abs{\sigma}(X^*,X)$ to the same limit. Moreover, every norm bounded martingale is fixed, i.e.~it has the form $(E_n^*x^*)$ for some $x^*\in X^*$.
\end{corollary}
\begin{proof} We show that $(E_n^*) $ is an abstract bistochastic filtration on $X^*$. Since $E_1^*$ is strictly positive, we have, by \cite[Lemma~5.3]{gx}, $E_1x_0=x_0$ for some weak unit $x_0>0$ in $X$. It is easily seen that $x_0$ acts a strictly positive order continuous functional on the order complete space $X^*$ and $E_1^{**}x_0=x_0$. Moreover, since $X$ is order continuous with a weak unit, there exists a strictly positive functional $x_1^*>0$ on $X$, by \cite[Theorem~4.15]{ali}. Clearly, $x_0^*:=E_1^*x_1^*=x_1^*\circ E_1$ is also strictly positive, and thus is a weak unit of $X^*$ by Lemma~\ref{spf}. Clearly, $E_1^*x_0^*=x_0^*$. This proves that $(E_n^*) $ is abstract bistochastic on $X^*$.\par

Let $(z_n^*)$ be a norm bounded submartingale relative to $(E_n^*)$. Then $(z_n^*) $ is uo-Cauchy in $X^*$, by \cite[Theorem~5.6]{gx}. Thus $(z_n^*)$ converges in uo and $\abs{\sigma}(X^*,X)$ to the same limit, by Theorem~\ref{wstar}. Suppose that $(x_n^*)$ is a norm bounded martingale. Then by what we have just proved, $x_n^*\xrightarrow{w^*}x^*$ for some $x^*\in X^*$. Fix any $k\geq 1$. For large $n$, we have $x_k^*=E_k^*x_n^*\xrightarrow{w^*} E_k^*x^*$. Therefore, $x_k^*=E_k^*x^*$ for all $k\geq 1$.
\end{proof}

\begin{remark}In contrast to \cite[Theorem~5.15]{gx}, we can not expect norm convergence for martingales in Theorem~\ref{mar}. The following example is modified from \cite[Example~6]{tro11}. Let $X=\ell_1$ and put $E_1$ to be the projection on $\ell_1$ consisting of $2\times 2$ diagonal blocks \begin{math}
  \Bigl[
  \begin{smallmatrix}
    \frac{1}{2} &\frac{1}{2}\\
 \frac{1}{2}     & \frac{1}{2}
  \end{smallmatrix}
  \Bigr]
\end{math}. Define $E_n$ by replacing the first $(n-1)$ diagonal blocks in $E_1$ with the identity matrix. Then $(E_n)$ is a (bounded) filtration on $\ell_1$. Clearly, $E_1$ and $E_1^*$ are both strictly positive. Consider the sequence $x_n^*=(1,-1,1,-1,\cdots)\in \ell_\infty$, where the first $2n$ coordinates are alternating $1$'s and $-1$'s, and all the other coordinates are $0$. It is easily seen that $(x_n^*)$ is a norm bounded martingale relative to $(E_n^*)$, but it is not norm convergent.
\end{remark}

\section{When does $w^*$-convergence imply uo-convergence?}

We now turn to study the spaces in whose dual space $w^*$-convergence implies $uo$-convergence. For this purpose, we need to use \cite[Theorem~3.1]{cw}. However, the implication (1)$\Rightarrow$(2) there fails in general.

\begin{example}[{\cite{wo}}]\label{nonato}Let's construct as follows a collection of closed line segments in the first quadrant of the plane, all of which start from the origin. Let $(d_n)$ be a strictly decreasing sequence in $(0,1]$ with $d_0=1$. Put $I_0$ be the unit interval along the positive $y$-axis. Take a sequence of line segments of length $d_1$ with slope strictly increasing to $+\infty$. We label this sequence as $\{I_{0,n}\}_1^\infty$ where $I_{0,n+1}$ has slope larger than the slope of $I_{0,n}$. Now for each $I_{0,n}$, we can construct another sequence of line segments of length $d_2$, whose slopes are strictly increasing to the slope of $I_{0,n}$ and such that if $n>1$, all of these line segments lie between $I_{0,n-1}$ and $I_{0,n}$. Continue in this manner using $d_3$ as the next length of line segments and so forth. Let $T$ be the union of all these line segments. Then $T$ is compact and Hausdorff. Let $X$ be the space of continuous functions on $T$ which vanish at the origin and are affine when restricted to each line segment. It is shown in \cite{wo} that $X$ is an AM-space whose dual space is lattice isometric to $\ell_1$.\par

Now let $\{t_n\}$ be the end points of the sequence of line segments which have length $d_2$ and lie between $I_{0,2}$ and $I_{0,3}$ with increasing slopes. Write $t_0=\lim_nt_n$. Consider the sequence $(\delta_{t_n})_1^\infty\subset X^*$, where $\delta_t$ is the point evaluation at $t$. Clearly, $\delta_{t_n}\xrightarrow{w^*}\delta_{t_0}$. If (1)$\Rightarrow$(2) in \cite[Theorem~3.1]{cw} were true, then $\wedge_1^\infty\abs{\delta_{t_n}-\delta_{t_0}}=0$. But one can easily see that $\delta_{t_n}\perp\delta_{t_0}$ for all $n\geq 1$. Thus,
$\wedge_1^\infty\abs{\delta_{t_n}-\delta_{t_0}}\geq \delta_{t_0}>0$, a contradiction.
\end{example}

Nevertheless, \cite[Theorem~3.1]{cw} still holds if the second condition there is replaced with ``for any order bounded sequence $(f_n)_1^\infty\subset E'$ with $f_n\rightarrow0$ in $\sigma(E',E)$ we have $\wedge_{n=1}^\infty \abs{f_n}=0$''. A simple modification of the original proof works.

\begin{lemma}\label{starse}Let $X$ be a Banach lattice. The following are equivalent:
\begin{enumerate}
\item\label{starsei1}for any sequence $(x_n^*)$ in $X^*$, if $x_n^*\xrightarrow{w^*}0$, then $x^*_n\xrightarrow{uo}0$.
\item\label{starsei2} for any sequence $(x_n^*)$ in $X^*$, if $x_n^*\xrightarrow{w^*}0$, then $\wedge_{n=1}^\infty\abs{x_n^*}=0$.
\end{enumerate}
In either case, $X^*$ is atomic.
\end{lemma}
\begin{proof}One can easily see that \eqref{starsei1} and \eqref{starsei2} both imply the third condition in \cite[Theorem~3.1]{cw}. Hence, in either case, $X^*$ is atomic. Now the desired equivalence is an immediate consequence of Lemma~\ref{sim2}.
\end{proof}

\begin{lemma}\label{samlim}Let $X$ be a Banach lattice. The following are equivalent:
\begin{enumerate}
\item\label{samlimi1} $X$ has order continuous norm,
\item\label{samlimi2} every norm bounded simultaneously uo- and $w^*$-convergent net in $X^*$ converges to the same limit.
\end{enumerate}
\end{lemma}
\begin{proof}\eqref{samlimi1}$\Rightarrow$\eqref{samlimi2} follows from Theorem~\ref{u2w}. Suppose \eqref{samlimi2} holds. If $X$ is not order continuous, then there exists a norm bounded disjoint sequence $(x_n^*)$ which does not $w^*$-converge to $0$; cf.~\cite[Corollary~2.4.3]{mn1991}. Therefore, we can find some subnet $(x_\alpha^*)$ of $(x_n^*)$ which $w^*$-converges to a non-zero functional $x^*\neq 0$ on $X$. But, by Lemma~\ref{sim}, $x_n^*\xrightarrow{uo}0$. Thus, $x_\alpha^*\xrightarrow{uo}0$. This yields a contradition.
\end{proof}

\begin{theorem}\label{w2u}Let $X $ be a Banach lattice. The following are equivalent:
\begin{enumerate}
 \item\label{w2ui1} $X$ is order continuous and atomic,
 \item\label{w2ui2} for every norm bounded net $(x_\alpha^*)$ in $X^*$, if $x_\alpha^*\xrightarrow{w^*}0$, then $x_\alpha^*\xrightarrow{uo}0$,
 \item\label{w2ui3} for every net $(x_\alpha^*)$ in $X^*$, if $x_\alpha^*\xrightarrow{w^*}0$, then $x_\alpha^*\xrightarrow{uo}0$.
 \end{enumerate}
\end{theorem}
\begin{proof}\eqref{w2ui3}$\Rightarrow$\eqref{w2ui2} is obvious. Suppose \eqref{w2ui2} holds. By Lemma~\ref{samlim}, $X$ is order continuous. By Lemma~\ref{starse}, $X^*$ is atomic. Hence, $X$ is also atomic. This proves \eqref{w2ui2}$\Rightarrow$\eqref{w2ui1}.\par

Now suppose \eqref{w2ui1} holds. Let $(x_\gamma)_{\gamma\in \Gamma}$ be a complete disjoint system of atoms of $X$. For each $\gamma$, define $x_\gamma^*\in X^*$ by putting $x_\gamma^*(x_\lambda)=1$ if $\lambda=\gamma$ and $0$ otherwise. It is clear that $(x_\gamma^*)$ is a complete disjoint system of atoms of $X^*$, and that $X^*$ embeds as an ideal into $\mathbb{R}^\Gamma$ with $x_\gamma^*$ corresponding to the function which takes value $1$ at $\gamma$ and $0 $ elsewhere. Now given any net $(x_\alpha^*) $ in $X^*$, if $x_\alpha^*\xrightarrow{w^*}0$, then $x_\alpha^*(x_\gamma)\rightarrow 0$ for all $\gamma\in \Gamma$. It follows that $(x_\alpha^*)$ converges pointwise to $0$ in $\mathbb{R}^\Gamma$. Equivalently, $x_\alpha^*\xrightarrow{uo}0$ in $\mathbb{R}^\Gamma$, and hence in $X^*$ by \cite[Lemma~3.4]{gx}.
This proves \eqref{w2ui1}$\Rightarrow$\eqref{w2ui3}.
\end{proof}

\section{Simultaneous uo- and $w^*$-convergence}
The following result is a dual version of \cite[Proposition~3.9]{gx} and establishes properties of sequences with simultaneous uo- and $w^*$-convergence.
\begin{proposition}\label{simul}
Let $X$ be a $\sigma$-order complete Banach lattice, and $(x_n^*)$ a $w^*$-convergent sequence in $X^*$. If $x_n^*\xrightarrow{uo}0$, then $x_n^*\xrightarrow{\abs{\sigma}(X^*,X)}0$.
\end{proposition}

\begin{proof}Fix any $x\in X_+$. For any $\varepsilon>0$, we have, by \cite[Theorem~4.42]{ali}, there exists $y^*\in X_+^*$ such that
$\abs{x_n^*}(x)-\abs{x_n^*}\wedge y^*(x)<\varepsilon$. Since $x_n^*\xrightarrow{uo}0$, we have $\abs{x_n^*}\wedge y^*\xrightarrow{o}0$, and thus, $\abs{x_n^*}\wedge y^*(x)\rightarrow0$. It follows that $\limsup_n\abs{x_n^*}(x)\leq\varepsilon$. Hence, $\lim_n\abs{x_n^*}(x)=0$.
\end{proof}

In view of Lemma~\ref{samlim}, one can not replace sequences with nets in Proposition~\ref{simul} in general. The following example shows that the condition that $X$ is $\sigma$-order complete can not be removed from the assumption of Proposition~\ref{simul}.

\begin{example}\label{counexa2} The example is taken from \cite[p.~762]{wnuk}. Let $X=C[0,1]$. Define $x_n^*=\delta_{\frac{1}{3n-3}}-\delta_{\frac{1}{3n-2}}+\delta_{\frac{1}{3n-1}}-\delta_{\frac{1}{3n-3}}$ for even $n$ and $x_n^*=\delta_{\frac{1}{3n-2}}-\delta_{\frac{1}{3n-1}}$ for odd $n$, where $\delta_t$ denotes the point evaluation at $t$. Then $(x_n^*)$ is a disjoint sequence in $X^*$. Since $X^*$ is order complete, $x_n^*\xrightarrow{uo}0$, by Lemma~\ref{sim}. It is clear that $x_n^*\xrightarrow{w^*}0$. But $\abs{x_{2n}^*}\xrightarrow{w^*}4\delta_0$ and $\abs{x_{2n+1}^*}\xrightarrow{w^*}2\delta_0$. Hence, $(\abs{x_n^*})$ does not converge in $w^*$. It follows that $(x_n^*)$ does not converge in $\abs{\sigma}(X^*,X)$.
\end{example}

We now characterize the spaces in whose dual space simultaneous uo- and $w^*$-convergence of sequences implies weak convergence. Recall that a Banach lattice $X$ is said to have the positive Grothendick property if every $w^*$-null positive sequence in $X^*$ is $w$-null.

\begin{proposition}\label{pg}Let $X$ be a $\sigma$-order complete Banach lattice. The following are equivalent:
\begin{enumerate}
\item\label{pgi1} $X$ has the positive Grothendick property,
\item\label{pgi2} for any sequence $(x_n^*)\subset X^*$, if $x_n^*\xrightarrow[uo]{w^*}0$, then $x_n^*\xrightarrow{w}0$,
\item \label{pgi3} for any sequence $(x_n^*)\subset X^*$, if $x_n^*\xrightarrow[uo]{w^*}0$, then $x_n^*\xrightarrow{\abs{\sigma}(X^*,X^{**})}0$.
\end{enumerate}
\end{proposition}

\begin{proof}
Assume \eqref{pgi1} holds. For $x_n^*\xrightarrow[uo]{w^*}0$, we have, by Proposition~\ref{simul}, $\abs{x_n^*}\xrightarrow{w^*}0$. Therefore, $\abs{x_n^*}\xrightarrow{w}0$ by the positive Grothendick property. This proves \eqref{pgi1}$\Rightarrow$\eqref{pgi3}. The implication \eqref{pgi3}$\Rightarrow$\eqref{pgi2} is clear.

Assume now \eqref{pgi2} holds. We first show that $X^*$ has order continuous norm. Since $X^*$ is order complete, it suffices to show that $X^*$ has $\sigma$-order continuous norm. To this end, let $x_n^*\downarrow0$. Then $x_n^*(x)\rightarrow0$ for any $x\in X$. Thus, $x_n^*\xrightarrow[o]{w^*}0$. It follows from assumption that $x_n^*\xrightarrow{w}0$. Now Dini's Theorem (\cite[Theorem~3.52]{ali}) implies $\norm{x_n^*}\rightarrow0$. This proves that $X^*$ has $\sigma$-order continuous norm.\par

Now let $(x_n^*)$ be a disjoint sequence in $X^*_+$ such that $x_n^*\xrightarrow{w^*}0$. By Lemma~\ref{sim}, we have $x_n^*\xrightarrow{uo}0$. Therefore, $x_n^*\xrightarrow{w}0$ by assumption. It follows from \cite[Theorem~5.3.13]{mn1991} that $X$ has the positive Grothendick property. This proves \eqref{pgi2}$\Rightarrow$\eqref{pgi1}.
\end{proof}

The following result characterizes the spaces in whose dual space simultaneous uo- and $w^*$-convergence of sequences implies norm convergence. Recall first that a Banach lattice $X$ is said to have the dual positive Schur property if every $w^*$-null positive sequence in $X^*$ is norm null. It is easily seen that $C(K)$-spaces have the dual positive Schur property and hence the positive Grothendick property.

\begin{proposition}\label{ps}Let $X$ be $\sigma$-order complete. The following are equivalent:
\begin{enumerate}
\item\label{psi1} $X$ has the dual positive Schur property
\item\label{psi2} for any sequence $(x_n^*)\subset X^*$, if $x^*_n\xrightarrow[uo]{w^*}0$, then $\norm{x_n^*}\rightarrow 0$.
\end{enumerate}
\end{proposition}

\begin{proof}\eqref{psi1}$\Rightarrow$\eqref{psi2} follows from Theorem~\ref{simul}.
\eqref{psi2}$\Rightarrow$\eqref{psi1} follows from \cite[Proposition~2.3]{wnuk} and Lemma~\ref{sim}.
\end{proof}

\begin{remark}
Observe that the condition that $X$ is $\sigma$-order complete is not needed for the implications \eqref{pgi3}$\Rightarrow$\eqref{pgi1} and \eqref{pgi2}$\Rightarrow$\eqref{pgi1} in Theorem~\ref{pg}, and is also not needed for \eqref{psi2}$\Rightarrow$\eqref{psi1} in Theorem~\ref{ps}.
All the reverse implications may fail without $\sigma$-order completeness of $X$. Indeed, let $X$ and $(x_n^*)$ be as in Example~\ref{counexa2}.
%It is well known and easily seen that $X$ satisfies the positive dual Schur property and the positive Grothendick property.
Recall that $x_n^*\xrightarrow[uo]{w^*}0$. We claim that $(x_n^*)$ does not converge weakly, and therefore, not in $\abs{\sigma}(X^*,X^{**})$ or in norm. Assume, otherwise, $(x_n^*)$ converges weakly. Then the limit must be $0$. It follows that $x_n^*\xrightarrow[uo]{w}0$. Note that $X^*$ is an AL-space, and hence has the positive Schur property. Therefore, $\norm{x_n^*}\rightarrow0$, by \cite[Theorem~3.11]{gx}, contradicting that $\norm{x_{2n}^*}=4$ for all $n\geq 1$.
\end{remark}

In general, we do not know about the situation for nets with simultaneous $w^*$- and uo-convergence. Below we consider such nets in the special case when $X$ has order continuous norm. Observe first that, in this case, in order for simultaneous $w^*$- and uo-convergence to imply norm convergence, $X$ must be finite-dimensional. Indeed, it follows from Proposition~\ref{ps} and \cite[Proposition~2.1]{wnuk}. For weak convergence, we have the following.

\begin{proposition}
Let $X$ be an order continuous Banach lattice. The following statements are equivalent.
\begin{enumerate}
\item\label{orpgi0}for any norm bounded net $(x_\alpha^*)$ in $X^*$, if $x_\alpha^*\xrightarrow[uo]{w^*}0$, then $x_\alpha^*\xrightarrow{w}0$,
\item\label{orpgi1} for any norm bounded net $(x_\alpha^*)$ in $X^*$, if $x_\alpha^*\xrightarrow{uo}0$, then $x_\alpha^*\xrightarrow{w}0$,
\item\label{orpgi2} $X$ has the positive Grothendick property,
\item\label{orpgi3} $X$ is reflexive.
\end{enumerate}
\end{proposition}
\begin{proof}The equivalence of \eqref{orpgi0} and \eqref{orpgi1} follows from Theorem~\ref{u2w}. The implication \eqref{orpgi3}$\Rightarrow$\eqref{orpgi1} follows from Theorem~\ref{u2w} again. The implication \eqref{orpgi1}$\Rightarrow$\eqref{orpgi2} follows from Proposition~\ref{pg}.\par

It remains to prove \eqref{orpgi2}$\Rightarrow$\eqref{orpgi3}.
Suppose that $X$ has the positive Grothendick property. We first claim that $X^*$ is order continuous, and therefore, is KB; cf.~\cite[Theorem~4.59]{ali}. Indeed, for any $x_n^*\downarrow 0$, we have $x_n^*\xrightarrow{w^*}0$. Thus, $x^*_n\xrightarrow{w}0$, by the positive Grothendick property. Dini's theorem implies that $\norm{x_n^*}\rightarrow0$. This proves that $X^*$ has $\sigma$-order continuous norm, and hence order continuous norm in view of its order completeness. We now claim that $X^{**}$ is also KB. Assume, otherwise, $X^{**}$ is not KB. Then $X^*$ has a lattice copy of $\ell_1$; cf.~\cite[Theorem~4.69]{ali}. Without loss of generality, assume $\ell_1\subset X^*$. Let $(e_n)$ be the standard basis of $\ell_1$. Then $e_n\xrightarrow{uo}0$ in $\ell_1$, hence in $X^*$ by \cite[Lemma~4.5]{gx}. By Theorem~\ref{u2w}, $e_n\xrightarrow{w^*}0$ in $X^*$, since $X$ is order continuous. It follows from the positive Grothendick property of $X$ that $e_n\xrightarrow{w} 0$ in $X^*$, hence in $\ell_1$, which is absurd. This proves the claim. By \cite[Theorem~4.70]{ali}, $X^*$ is reflexive and hence so is $X$.
\end{proof}

The equivalence of \eqref{orpgi2} and \eqref{orpgi3} also follows from \cite[Theorem~5.3.13]{mn1991}. The preceding arguments demonstrate an interesting application of uo-convergence by providing an alternative proof of this equivalence.

\end{document}